\documentclass{amsart}
\pdfoutput=1
\usepackage[utf8]{inputenc}
\usepackage{amsmath}
\usepackage{amssymb}
\usepackage{amsthm}
\usepackage{tabularx}
\usepackage{hyperref}
\usepackage{tikz}
\usetikzlibrary{cd}
\everymath{\displaystyle}
\usepackage[hmargin = 1in,vmargin=1in]{geometry}
\newcommand{\Sym}{\mathrm{Sym}}
\newcommand{\diag}{\mathrm{diag}}
\newcommand{\Spec}{\mathrm{Spec}}

\newcommand{\GL}{\mathrm{GL}}

\newcommand{\semi}{_\mathrm{ss}}

\newcommand{\der}{_\mathrm{der}}

\newcommand{\lin}{_{\mathrm{lin}}}
\newcommand{\ab}{_{\mathrm{ab}}}
\newcommand{\tr}{\mathrm{tr}}

\newcommand{\red}{_{\mathrm{red}}}

\newcommand{\Het}{\mathrm{H}_{\mathrm{\acute{e}t}}}

\newcommand{\Hom}{\mathrm{Hom}}

\newcommand{\Z}{\mathbb{Z}}
\newcommand{\F}{\mathbb{F}}

\newcommand{\Q}{\mathbb{Q}}

\newcommand{\G}{\mathbb{G}}

\newcommand{\Ext}{\mathrm{Ext}}

\newcommand{\Id}{\mathrm{Id}}

\newcommand{\Glin}{G_{\mathrm{lin}}}

\newcommand{\Gab}{G_{\mathrm{ab}}}
\newcommand{\Gant}{G_{\mathrm{ant}}}

\newcommand{\Grp}{\mathrm{Grp}}

\newcommand{\Ho}{\mathrm{H}}
\newcommand{\chr}{\mathrm{char}}
\newcommand{\inv}{\mathrm{inv}}
\newcommand{\Ru}{\mathrm{R}_{\text{u}}}
\newcommand{\R}{\mathrm{R}}

\usepackage{xfrac}

\usepackage{scalerel}
\usepackage{stackengine,wasysym}

\newcommand{\Fix}{\mathrm{Fix}}

\usepackage[hmargin = 1in,vmargin=1in]{geometry}
\usepackage[backend=biber, sorting=anyvt]{biblatex}
\renewbibmacro{in:}{%
  \ifentrytype{article}{}{\printtext{\bibstring{in}\intitlepunct}}}

\newtheorem{theorem}{Theorem}
\newtheorem*{theorem*}{Theorem}
\newtheorem{strategy}[theorem]{Strategy}
\newtheorem{corollary}[theorem]{Corollary}
\newtheorem{lemma}[theorem]{Lemma}
\newtheorem{prop}[theorem]{Proposition}
\theoremstyle{definition}
\newtheorem{definition}[theorem]{Definition}
\newtheorem*{definition*}{Definition}
\newtheorem{example}[theorem]{Example}

\newtheorem*{question*}{Question}

\theoremstyle{remark}
\newtheorem{remark}[theorem]{Remark}

\usepackage{blindtext}
\usepackage{subfiles}
\newcommand{\J}{\mathrm{J}}

\title{Functorial splitting of $l$-adic cohomology\\
of an extension of group varieties}
\author{Victor de Vries}

\begin{document}

\maketitle

\begin{abstract}
In this document we consider an exact sequence of group varieties $e\to N\to G\to Q\to$~$e$ over an algebraically closed field. We show that for $l\neq \chr(k)$ a prime there exists an isomorphism of graded $\Q_l$-algebras $\Het^*(G,\Q_l)\cong \Het^*(N,\Q_l)\otimes_{\Q_l}\Het^*(Q,\Q_l)$ that is compatible with pullback homomorphisms $\varphi^*$ of endomorphisms $\varphi:G\to G$ that stabilize $N$.
\end{abstract}

\section{Introduction}
Let $G$ be an algebraic group over an algebraically closed field $k$. We will call it a group variety if it is smooth and connected. For $l\neq \chr(k)$ a prime the $l$-adic cohomology ring of a variety~$X$ is denoted $\Ho^*(X):=\Het^*(X,\Q_l)=\bigoplus_{r\geq 0}\Het^r(X,\Q_l)$ and we denote the tensor product over $\Q_l$ by~$\otimes$. For an exact sequence of group varieties, we study the property that we call $(*)$ related to the $l$-adic cohomology of the involved group varieties:

\begin{definition*}
An exact sequence of group varieties $e\to N\overset{\iota}{\to} G\overset{\pi}{\to} Q\to e$ has $(*)$ if there is an isomorphism of graded $\Q_l$-algebras $s\otimes \pi^*:\Ho^*(N)\otimes\Ho^*(Q)\to \Ho^*(G)$ where $s$ is a section to $\iota^*$.
\end{definition*}

Our main theorem asserts that any exact sequence of group varieties has this property.

\begin{theorem*}\label{result}
An exact sequence of group varieties $e\to N\to G\to Q\to e$ has $(*)$.
\end{theorem*}

The way that we prove the theorem is as follows: By first showing that the property $(*)$ holds for certain specific exact sequence. For example a group variety $G$ fits into an exact sequence $e\to \Glin\to G\to \Gab\to e$, where $\Glin$ is linear and $\Gab$ is an abelian variety and we show that this sequence has $(*)$.

Then we introduce the following generalization of an exact sequence.

\begin{definition*}
A sequence of group varieties $e\to N\overset{\iota}{\to} G\overset{\pi}{\to} Q\to e$ is called \textit{almost exact} if either:
\begin{itemize}
    \item There exists an exact sequence $e\to M\overset{j}{\to} G\overset{\pi}{\to} Q\to e$ is exact and an isogeny $q:N\to M$ such that $\iota=j\circ q$.
    \item There exists an exact sequence $e\to N\overset{\iota}{\to} G\overset{p}{\to} P\to e$ and an isogeny $q:P\to Q$ is such that $\pi=q\circ p$.
\end{itemize}
We also say that an almost exact sequence has $(*)$ if it satisfies the condition in Definition \ref{star}.
\end{definition*}

The use of this definition lies in the fact that whenever $e\to N\to G\to Q\to e$ is exact, the induced sequences $e\to N\lin\to G\lin\to Q\lin\to e$ and $e\to N\ab\to G\ab\to Q\ab\to e$ will generally not be exact but will be almost exact.\\
After showing that exact sequences $e\to N\to G\to Q\to e$, where either $N,G,Q$ are all linear group varieties or are all abelian varieties, have $(*)$, we show that the almost exact sequences $e\to N\lin\to G\lin\to Q\lin\to e$ and $e\to N\ab\to G\ab\to Q\ab\to e$ both have $(*)$. This will allow us to conclude that the exact sequence $e\to N\to G\to Q\to e$ has $(*)$.

One use of this theorem lies in the following: Let $\sigma:G\to G$ be an endomorphism that stabilizes the subgroup variety $N$ of $G$. Writing $\sigma_N:N\to N$ and $\sigma_Q:Q\to Q$ for the induced endomorphisms, one sees that through the isomorphism in Theorem \ref{result}, $\sigma^*$ corresponds to $\sigma_N^*\otimes \sigma_Q^*$. This implies that we can calculate the graded trace for $\sigma$, $\tr(\sigma^*):=\sum_{r\geq 0}(-1)^r\tr(\sigma^*\,|\,\Ho^r(G))$ by $\tr(\sigma)=\tr(\sigma_N)\cdot \tr(\sigma_Q)$. An application of our result is the following.

The authors of \cite{BCH} studied the fixed point count of the $n$'th iterate of $\sigma$ and came up with the following formula (\cite{BCH} Theorem 7.2.1), which holds in the case that $k=\overline{\F}_p$:
\begin{equation*}
    \#\Fix(\sigma^n)=|d_n|c^nr_n|n|_p^{s_n}p^{-t_n|n|_p^{-1}}
\end{equation*}
Here we have $c\in \Q_{>0}$ and we have that $s_n,t_n\in \Z_{>0}$ and $r_n\in \Q_{>0}$ are periodic. The term $d_n$ may be defined even when $k\neq \overline{\F}_p$ for any $p$ (see \cite{BCH} Definition 12.3.2) and can be given a cohomological interpretation as follows: The group variety $G$ fits into an exact sequence $e\to \Glin\to G\to \Gab\to e$ where $\Glin$ is a linear and fully characteristic subgroup variety and $\Gab$ is an abelian variety. The term $d_n$ equals $\tr(\sigma\lin^n)\cdot \tr(\sigma\ab^n)$, where $\sigma\lin$ is the induced endomorphism on $\Glin$ and $\sigma\ab$ the one on $\Gab$.\\

One may ask the natural question, whether $d_n$ is equal to $\tr(\sigma^n)$? The discussion below Theorem~\ref{result} shows that this is indeed the case. In the case that $k=\overline{\F}_p$ this isomorphism was shown to exist by the authors of \cite{BCH} by using the result of Arima (\cite{Arima} Theorem 1), which says that there exists an isogeny $\Glin\times_{\overline{\F}_p}\Gab\to G$. However over a general algebraically closed field $k$, such an isogeny need not exist, showing the need to use a different approach to tackle this problem in general.\\

\section{Generalities and strategy}

First we define the property of group varieties in which we are interested.

\begin{definition}\label{star}
An exact sequence of group varieties $e\to N\overset{\iota}{\to} G\overset{\pi}{\to} Q\to e$ has $(*)$ if there is an isomorphism of graded $\Q_l$-algebras $s\otimes \pi^*:\Ho^*(N)\otimes\Ho^*(Q)\to \Ho^*(G)$ where $s$ is a section to $\iota^*$.
\end{definition}

The aim of this document is to show that any group variety has this property and hence to prove the following theorem.

\begin{theorem}
An exact sequence of group varieties $e\to N\to G\to Q\to e$ has $(*)$.
\end{theorem}

We now state two known lemmas that are key for proving Theorem \ref{result}. The first one is a structure theorem on algebraic groups (see \cite{milne_AG}).

\begin{lemma}
The following statements about group varieties are true:
\begin{itemize}
\item A group variety $G$ is an extension of an abelian variety $\Gab$ by a linear group variety $\Glin$.
\item A linear group variety is the extension of a reductive group variety $G\red$ by a unipotent group variety $\Ru(G)$.
\item A reductive group variety $G$ is the extension of a semisimple group variety $G\semi$ by a torus $\R(G)$
\item A group variety $G$ has a largest normal subgroup variety $\Gant$, called the \textit{anti-affine group variety} of $G$, such that $G/\Gant$ is affine.
\end{itemize}
\end{lemma}

The second lemma says that an isogeny induces an isomorphism on cohomology, which is a well known fact.

\begin{lemma}\label{isogeny}
Let $\varphi:G\to H$ be an isogeny of group varieties. Then $\varphi^*:\Ho^*(H)\to \Ho^*(G)$ is an isomorphism.
\begin{proof}
It suffices to prove the statement for separable and purely inseparable isogenies. In the purely inseparable case, $\ker(\varphi)$ is connected, so $\Z/l^n\Z\to \varphi_*\varphi^*\Z/l^n\Z$ is an isomorphism of sheaves. Since $\ker(\varphi)$ is zero dimensional, the Leray spectral sequence map $\Ho^n(H,\varphi_*\varphi^*\Z/l^n\Z)\to \Ho^n(G,\varphi^*\Z/l^n\Z)$ is an isomorphism, which gives the result.\\
In the separable case, $\ker(\varphi)$ may be thought of as a finite group $M$ and $H$ is identified with $G/M$. For any abelian sheaf $\mathcal{F}$ the map $\Ho^*(H,\mathcal{F})\to \Ho^*(G,\varphi^*\mathcal{F})\cong \Ho^*(H,\varphi_*\varphi^*\mathcal{F})$ is induced by $\mathcal{F}\to \varphi_*\varphi^*\mathcal{F}$, which identifies $\mathcal{F}=(\varphi_*\varphi^*\mathcal{F})^M$, see (\cite{srinivasan2006representations} (5.10)). Then there exists a map $\varphi_*\varphi^*\mathcal{F}\to \mathcal{F}$ given by $x\mapsto \sum_{m\in M}m\cdot~x$, called the transfer map. Applying this to the sheaves $\Z/l^n\Z$ and passing to $l$-adic cohomology gives a map $N:\Ho^*(G)\to \Ho^*(H)$ such that $N\circ \varphi^*$ is multiplication by $\#M$ and $\varphi^*\circ N$ maps $x\mapsto \sum_{m\in M}m\cdot x$. This first point gives that $\varphi^*$ is injective and the second point gives that under $\varphi^*$, $\Ho^*(H)$ is identified with $\Ho^*(G)^M$. \\
Since the action of $M$ on $\Ho^*(G)$ is induced by the one of $G(k)$ and $G$ is a connected algebraic group, the action is trivial on compactly supported cohomology (\cite{Deligne-Lustzig} Proposition 6.4). When combined with the functoriality of the Poincaré duality pairing, this gives that $M$ acts trivially on $\Ho^*(G)$, which gives the result.
\end{proof}
\end{lemma}

In order to share the strategy for the proof, we need to make one more definition.

\begin{definition}\label{almostE}
A sequence of group varieties $e\to N\overset{\iota}{\to} G\overset{\pi}{\to} Q\to e$ is called \textit{almost exact} if either:
\begin{itemize}
    \item There exists an exact sequence $e\to M\overset{j}{\to} G\overset{\pi}{\to} Q\to e$ is exact and an isogeny $q:N\to M$ such that $\iota=j\circ q$.
    \item There exists an exact sequence $e\to N\overset{\iota}{\to} G\overset{p}{\to} P\to e$ and an isogeny $q:P\to Q$ is such that $\pi=q\circ p$.
\end{itemize}
We also say that an almost exact sequence has $(*)$ if it satisfies the condition in Definition \ref{star}.
\end{definition}

Now we give the strategy for proving the main thoerem.

\begin{strategy}\label{strategy}
The strategy for proving Theorem \ref{result} will be as follows:
\begin{enumerate}
    \item We show that the exact sequence $e\to \Glin\to G\to \Gab\to e$ has $(*)$. For this we make use of the existence of the anti-affine group variety.  We also show that for $G$ linear the exact sequence $e\to \Ru(G)\to G\to G\red\to e$ and for $G$ reductive the exact sequence $e\to \R(G)\to G\to G\semi\to e$ both have $(*)$.
    \item We show that an exact sequence of abelian varieties (resp. tori, resp. semisimple group varieties) has $(*)$.
    \item We show that certain almost exact sequences such as $e\to N\lin\to G\lin \to Q\lin \to e$ have $(*)$ and in the end we conclude that any exact sequence $e\to N\to G\to Q\to e$ has $(*)$.
\end{enumerate}
\end{strategy}

All of the steps will use Lemma \ref{isogeny}.\\
We know quite a bit about $\Ho^*(G)$ for $G$ a general group variety. For $\mu:G\times_k G\to G$ the multiplication, $\inv:G\to G$ inversion and $e:\Spec(k)\to G$ the identity, the data $(\Ho^*(G),\mu^*,\inv^*,e^*)$ defines a $\Q_l$-Hopf-algebra, where one uses the Künneth isomorphism $p_1^*\cup p_2^*:\Ho^*(G)\otimes \Ho^*(G)\to \Ho^*(G\times_k G)$   (see \cite{hopfalgebra} for an introduction to Hopf-algebras). An important subspace of a Hopf-algebra is the primitive subspace.

\begin{definition}
Let $(\Ho, \Delta, \iota, \epsilon)$ be a Hopf algebra over a field $l$. Then the \textit{primitive subspace} of $\Ho$ is defined by $\mathrm{P}\Ho:=\{x\in \Ho\,|\,\Delta(x)=x\otimes 1+1\otimes x\}$.
\end{definition}

Note that a homomorphism of Hopf-algebras preserves the primitive subspace. Since $\Ho^*(G)$ is finite dimensional over $\Q_l$ and graded-commutative we can use the following theorem of Hopf.

\begin{theorem} [{Part 2.5D \cite{Cartier2007}}] \label{hopfstructuur}
Let $\Ho^*$ be a graded commutative Hopf-algebra over a field $l$ with $\chr(l)=0$ such that $\Ho^n$ is finite dimensional for all $n\geq 0$. Then there is a natural isomorphism of graded $k$-algebras $\Ho^*=\bigwedge^* \mathrm{P}\Ho^*:=\bigoplus_{r\geq 0} \bigwedge^r \mathrm{P}\Ho^*$, where for $\{z_{i_j}\}_{1\leq j\leq r}$ homogeneous elements of $\mathrm{P}\Ho^*$ of degree $i_j$, the element $z_{i_1}\wedge...\wedge z_{i_r}$ is homogeneous of degree $\sum_{j=1}^r i_j$.
In particular we have $\Ho^*(G)=\bigwedge^* \mathrm{P}\Ho^*(G)$.
\end{theorem}

Using the primitive subspace allows us to formulate the property $(*)$ in a simpler way.

\begin{lemma} \label{primitive}
Let $e\to N\overset{\iota}{\to} G\overset{\pi}{\to} Q\to e$ be an exact sequence of group varieties. The sequence has~$(*)$ if and only if $0\to \mathrm{P}\Ho^*(Q)\overset{\pi^*}{\to} \mathrm{P}\Ho^*(G)\overset{\iota^*}{\to} \mathrm{P}\Ho^*(N)\to 0$ is an exact sequence of graded $\Q_l$-vector spaces.
\begin{proof}
If the sequence of primitive subspaces is exact, then any section $s$ to $\iota^*$ induces an isomorphism: \begin{equation*}\Ho^*(N)\otimes \Ho^*(Q)=\bigwedge^*\mathrm{P}\Ho^*(N)\otimes \bigwedge^*\mathrm{P}\Ho^*(Q)\overset{\pi^*\otimes s}{\to} \bigwedge^*\mathrm{P}\Ho^*(G)=\Ho^*(G)
\end{equation*}
Conversely if the sequence has $(*)$, then note that pullback homomorphisms give Hopf-algebra homomorphisms on cohomology, which preserve primitive subspaces. It follows that $\iota^*:\mathrm{P}\Ho^*(G)\overset{\iota^*}{\to} \mathrm{P}\Ho^*(N)$ is surjective and that $\pi^*:\mathrm{P}\Ho^*(Q)\overset{\pi^*}{\to} \mathrm{P}\Ho^*(G)$ is injective. Moreover $\dim \mathrm{P}\Ho^*(Q)+ \dim \mathrm{P}\Ho^*(N)= \mathrm{P}\Ho^*(G)$ and $\ker(\pi^*)\subset \mathrm{Im}(\pi^*)$, so we obtain an exact sequence of primitive subspaces.
\end{proof}
\end{lemma}

\section{The proof}

Now we move towards the proof of Theorem \ref{result} as described in Strategy \ref{strategy}.

\subsection*{Step 1}
First we deal with the case that $G$ is commutative, so $G$ will be a commutative group variety until we say that this is no longer the case. Let $\mathcal{C}$ denote the category of commutative algebraic groups.

\begin{prop}\label{Ext and H1}
For all $n\geq 1$ there is an isomorphism of functors $\Ext(-,\Z/l^n\Z)\to \Ho^1(-,\Z/l^n\Z)$, where both are seen as functors $\mathcal{C}\to \Grp$.
\end{prop}
\begin{proof}
That there is a natural transformation is well-known and may be found in \cite{Serre} Chapter 7. Miyanishi (\cite{Miyanishi} Lemma 4) shows that it is in fact an isomorphism.
\end{proof}

\begin{lemma}\label{H1}
There is an exact sequence $0\to \Ho^1(\Gab)\overset{\pi^*}{\to} \Ho^1(G)\overset{\iota^*}{\to} \Ho^1(\Glin)\to 0$ of $\Q_l$-vector spaces.
\begin{proof}
Since $\Q_l$ is flat over $\Z_l$, it suffices to prove it with $\Z_l$ coefficients. As the groups $\Ho^1(G,\Z/l^n\Z)$ are finite, we obtain that passing to the limit preserves exactness by the Mittag-Leffler condition. So it suffices to show that the above statement holds with $\Z/l^n\Z$ coefficients. By Proposition \ref{Ext and H1} this translates into a statement about $\Ext(-,\Z/l^n\Z)$, which is true by Theorem 7.13 in \cite{Serre}.
\end{proof}
\end{lemma}

We can now prove the following proposition.

\begin{prop}\label{commutative}
When $G$ is commutative the sequence $e\to \Glin \to G\to G\ab\to e$ has $(*)$.
\begin{proof}
It suffices to show that $\Ho^1(G)=\mathrm{P}\Ho^*(G)$ by Lemma \ref{primitive} and Lemma \ref{H1}. So we need to show that for all $x\in \Ho^1(G)$, we have $\mu^*(x)=p_1^*(x)+p_2^*(x)$. On the degree $1$ part a section to the Künneth isomorphism $\Ho^*(G\times_k G)\to \Ho^*(G)\otimes \Ho^*(G)$ is given by $x\mapsto \iota_1^*(x)\otimes 1+1\otimes \iota_2^*(x)$ for $\iota_1:G\to G\times_k \{e\}$ and $\iota_2:G\to \{e\}\times_k G$. So for $x\in \Ho^1(G)$ we have that $\iota_1^*\mu^*(x)\otimes 1+1\otimes \iota_2^*\mu^*(x)=x\otimes 1+1\otimes x$. Applying $p_1^*\cup p_2^*$ to this element shows $\mu^*(x)=p_1^*(x)+p_2^*(x)$ and hence $\Ho^1(G)\subset  \mathrm{P}\Ho^*(G)$. Now by Theorem \ref{hopfstructuur} it suffices to show that:\begin{equation*} \label{cd}\Ho^r(G)=0 \text{ for all } r>\dim \Ho^1(G) \end{equation*}
Since $\Glin$ is commutative it is an extension of a torus $T$ by a unipotent group $\Ru(\Glin)$. Let $H=G/\Ru(\Glin)$. As $\Ru(\Glin)$ is isomorphic to an affine space, it follows that $\Ho^*(H)\to \Ho^*(G)$ is an isomorphism by looking at the leray spectral sequence, so it suffices to prove (\ref{cd}) for $H$ and so we may assume that $G$ is the extension of an abelian variety by a torus. In this case we have $\Ho^1(\Glin)=\Q_l^n$ and $\Ho^1(\Gab)=\Q_l^{2g}$ where $n=\dim(\Glin)$ and $g=\dim(\Gab)$. Since the abelian variety and torus have cohomological dimension $2g$ and $n$ respectively, the Leray spectal gives that $\Ho^r(G)=0$ for all $r>2g+n$. By Lemma \ref{H1} we have $\dim\Ho^1(G)=2g+n$, which gives that (\ref{cd}) holds and thus the proposition holds.
\end{proof}
\end{prop}

From now on we no longer require $G$ to be commutative.

\begin{prop}\label{dimension}
Let $G$ be a group variety and consider $e\to \Glin\to G\to \Gab\to e$. There exists \textbf{an} isomorphism of graded $\Q_l$-vector spaces $\mathrm{P}\Ho^*(\Glin)\oplus \mathrm{P}\Ho^*(\Gab)\cong \mathrm{P}\Ho^*(G)$.
\begin{proof}
Let $D:=\Gant$, which is a commutative group variety. We have an exact sequence of algebraic groups $0\to \ker\overset{\iota}{\to} D\times \Glin\overset{m}{\to} G\to 0$, where $\ker\cong D\cap \Glin $ and the map $D\cap \Glin\to D\times \Glin$ is given by $a\mapsto (-a,a)$. This gives rise to the following commuting diagram:
\begin{equation}\label{commuting diagram}
\begin{tikzcd}
e \arrow[r] & D\lin \arrow[r, "j'"] \arrow[d, "\iota\lin", hook] & D\times_k \Glin \arrow[r, "m'"] \arrow[d, "\Id"] & Q \arrow[d, "q", two heads] \arrow[r] & e \\
e \arrow[r] & D\cap \Glin \arrow[r, "j"]                         & D\times_k \Glin \arrow[r] \arrow[r, "m"]         & G \arrow[r]                             & e
\end{tikzcd}
\end{equation}

By Proposition \ref{commutative} the embedding $-\iota_D:D\lin\to D$ induces a surjection $\Ho^*(D)\to \Ho^*(D\lin)$. Moreover $\Ho^*(D\times_k \Glin)\to \Ho^*(D)$ is surjective by Künneth, and hence the pullback $(j')^*:\Ho^*(D\times_k \Glin)\to \Ho^*(D\lin)$ is surjective.
Since $D_{\text{lin}}$ is a commutative linear group variety, the fiber bundle $D\times_k \Glin\to Q$ is Zariski locally trivial by Proposition 7.6 in \cite{Serre}. \\
These two statements imply that the Leray-Hirsch principle applies (see \cite{husemoller1994fibre} p.245 for a proof in topology, which carries over to our case since the fibration is Zariski locally trivial implying that the Mayer-Vietoris argument of the reference can be used). This gives an isomorphism $\Ho^*(\Glin\times_k D)\cong \Ho^*(Q)\otimes \Ho^*(D_{\text{lin}})$ as $\Ho^*(Q)$-modules. By the Künneth formula we have $\Ho^*(\Glin\times_k D)\cong \Ho^*(\Glin)\otimes \Ho^*(D)$ and by Proposition \ref{commutative} we obtain that:\begin{equation} \label{iso}\Ho^*(Q)\otimes \Ho^*(D_{\text{lin}})\cong \Ho^*(\Glin)\otimes \Ho^*(D_{\text{lin}})\otimes \Ho^*(D\ab)\end{equation}

Since $D\lin$ has finite index in $D\cap \Glin$, the map $q$ is an isogeny. This map induces a homomorphism $q\ab:Q\ab\to G\ab$. Since $G=D\cdot \Glin$ it also follows that $q\ab$ is an isogeny. An isogeny induces an isomorphism on the cohomology. This means that the isomorphism in (\ref{iso}) reads:
\begin{equation*}
    \Ho^*(G)\otimes \Ho^*(D_{\text{lin}})\cong \Ho^*(\Glin)\otimes \Ho^*(\Gab)\otimes \Ho^*(D\lin)
\end{equation*}

Since $\Ho^*(D_{\text{lin}})$ is a nonzero graded $\Q_l$-vector space we see that there is an isomorphism of graded $\Q_l$-vector space isomorphism $\Ho^*(G)\cong \Ho^*(\Gab)\otimes \Ho^*(\Glin)$. Hence $\bigwedge^* \mathrm{P}\Ho^*(G)\cong \bigwedge^* (\mathrm{P}\Ho^*(\Glin)\oplus \mathrm{P}\Ho^*(\Gab))$ and so we obtain that: \begin{equation*}\mathrm{P}\Ho^*(G)\cong \mathrm{P}\Ho^*(\Glin)\oplus \mathrm{P}\Ho^*(\Gab)\end{equation*}
\end{proof}
\end{prop}

Now we will use the above to show that the desired isomorphism indeed exists.

\begin{theorem}\label{chevalley}
There is a section $s$ to $\iota^*:\Ho^*(G)\to \Ho^*(\Glin)$ such that \\$s\otimes \pi^*:\Ho^*(\Glin)\otimes \Ho^*(\Gab)\to \Ho^*(G)$ is an isomorphism of graded $\Q_l$-algebras.
\begin{proof}
We will show that $0\to \mathrm{P}\Ho^*(\Gab)\overset{\pi^*}{\to}\mathrm{P}\Ho^*(G)\overset{\iota^*}{\to} \mathrm{P}\Ho^*(\Glin)\to 0$ is exact, which is sufficient by Lemma~\ref{primitive}. Since $\dim \mathrm{P}\Ho^*(\Glin)+\dim \mathrm{P}\Ho^*(\Gab)=\dim \mathrm{P}\Ho^*(G)$ and since $\ker(\pi^*)\subset \mathrm{Im}(\iota^*)$ it will suffice to show that $\iota^*$ is surjective and $\pi^*$ is injective.\\
For this we look at (\ref{commuting diagram}) and we stick to the notation of Proposition \ref{dimension}. For $p_1:D\times_k\Glin\to D$ the first projection we have $\pi\circ m=q\ab\circ \pi_D\circ p_1$. This implies that $m^*\pi^*=p_1^*\pi_D^*q\ab^*$. We have that $m^*$ is injective by the Leray-Hirsch principle and $\pi_D^*$ is injective by Proposition \ref{commutative}. Moreover, $p_1^*$ and $q\ab^*$ are also injective and hence $\pi^*$ is injective.\\
To show that $\iota^*$ is surjective, pick $x\in \Ho^*(\Glin)$ and write $p_2:D\times_k\Glin\to \Glin$ for the second projection. Then the Leray Hirsch principle applied to $m':D\times_k\Glin\to Q$ together with the fact that $q\circ m'=m$ implies that $p_2^*x=\sum m^*\beta_n\cup \alpha_r^{\sigma_r}$ for $\beta_n\in \Ho^n(G)$ and the $\alpha_r^{\sigma_r}$ such that $(j')^*\alpha_r^{\sigma_r}$ generate $\Ho^r(D\lin)$ freely. Notice that the $\alpha_r^{\sigma_r}$ can be chosen of the form $p_1^*\beta_r^{\sigma_r}$ for $\beta_r^{\sigma_r}\in \Ho^r(D)$. Then note that by using that $\iota_2^*p_2^*$ is the identity, we have $p_2^*x=p_2^*\iota_2^*p_2^*x=\sum p_2^*\iota_2^*m^*\beta_n\cup p_2^*\iota_2^*p_1^*\beta_r^{\sigma_r}=p_2^*\iota_2^*m^*y$ for $y\in \Ho^*(G)$. So this gives $m\circ\iota_2=\iota$ and since $p_2^*$ is injective, this gives $x=\iota^*y$ and hence $\iota^*$ is surjective. We conclude that the sequence of primitive subspaces is exact and that hence the theorem is true.
\end{proof}
\end{theorem}

We still have to show that for linear $G$ the sequence $e\to \Ru(G)\to G\to G\red\to e$ and for reductive $G$ the sequences $e\to R(G)\to G\to G\semi\to e$ give a tensor product decomposition of $l$-adic cohomology. The first one follows trivially as $\Ru(G)$ is isomorphic to an affine space and hence the Leray spectral sequence gives an isomorphism $\Ho^*(G\red)\to \Ho^*(G)$. For the second one we have the following lemma:

\begin{lemma}\label{isogenydiag}
Consider the commuting diagram of group varieties with exact rows and all columns isogenies:
\begin{equation*}
\begin{tikzcd}
e \arrow[r] & N_1 \arrow[r, "\iota_1"]                  & G_1 \arrow[r, "\pi_1"]                  & Q_1 \arrow[r]                  & e \\
e \arrow[r] & N_2 \arrow[r, "\iota_2"] \arrow[u, "q_N"] & G_2 \arrow[r, "\pi_2"] \arrow[u, "q_G"] & Q_2 \arrow[r] \arrow[u, "q_Q"] & e
\end{tikzcd}
\end{equation*}
Then the top row has $(*)$ if and only if the bottom row has $(*)$.
\begin{proof}
Upon taking primitive subspaces of $l$-adic cohomology, we get the following commuting diagram:
\begin{equation*}
\begin{tikzcd}
0 & \mathrm{P}\Ho^*(N_1) \arrow[l] \arrow[d, "q_N^*"]       & \mathrm{P}\Ho^*(G_1) \arrow[l] \arrow[d, "q_G^*"]  & \mathrm{P}\Ho^*(Q_1) \arrow[d, "q_Q^*"] \arrow[l]      & 0 \arrow[l] \\
0 & \mathrm{P}\Ho^*(N_2) \arrow[l] & \mathrm{P}\Ho^*(G_2)  \arrow[l] & \mathrm{P}\Ho^*(Q_2)  \arrow[l] & 0 \arrow[l]    
\end{tikzcd}
\end{equation*}
Since $q_N^*,q_G^*$ and $q_Q^*$ are isomorphisms, the top sequence is exact if and only if the bottom sequence is exact.
\end{proof}
\end{lemma}

This gives the following immediate corollary:

\begin{corollary}\label{reduct}
For reductive $G$, the sequence $e\to R(G)\to G\to G\semi\to e$ has $(*)$.
\begin{proof}
Since $G$ is reductive, $m:R(G)\times_k G\der\to G$ is an isogeny. Hence there is a commuting diagram: 
\begin{equation*}
\begin{tikzcd}
e \arrow[r] & R(G) \arrow[r]                             & G \arrow[r]                                        & G\semi \arrow[r]          & e \\
e \arrow[r] & R(G) \arrow[r, "\iota_1"] \arrow[u, "\Id"] & R(G)\times_k G\der \arrow[r, "p_2"] \arrow[u, "\mu"] & G\der \arrow[r] \arrow[u] & e
\end{tikzcd}
\end{equation*}
Now the previous lemma gives the result as the bottom sequence is split.
\end{proof}
\end{corollary}

This concludes step 1 of Strategy \ref{strategy}.

\subsection*{Step 2}
In this step we show that an exact sequence of abelian varieties, unipotent groups, tori or semisimple groups give a tensor product decomposition of $l$-adic cohomology. Note that the case of unipotent groups is trivial as they have no higher cohomology. For the other cases we have the following lemma.

\begin{lemma}\label{types}
An exact sequence $e\to N\to G\to Q\to e$ of abelian varieties (resp. semisimple group varieties, resp. tori) has $(*)$.
\begin{proof}
For any abelian variety $A$, the simple abelian subvarieties $A_1,...,A_n$ of $A$ satisfy that multiplication $A_1^{m_1}\times_k...\times_k A_n^{m_n}\to A$ is an isogeny by p.42 of \cite{milneAV}. So there exist abelian varieties $M$ and $H$ such that multiplication $M\times_k H\to G$ is an isogeny. The result for abelian varieties now follows from Lemma \ref{isogenydiag}. \\
For any semisimple group variety $G$ there also exist almost simple subgroup varieties $G_1,...,G_n$ such that multiplication $G_1\times_k...\times_k G_n\to G$ is an isogeny by Theorem 21.51 of \cite{milne_AG}, so the proof is the same as above.\\
For tori, the category of tori over an algebraically closed field is contravariantly equivalent to the category of torsionfree finitely generated abelian groups, the functor being $T\mapsto T^*:=\Hom(T,\G_m)$. So an exact sequence of tori $e\to N\to G\to e$ splits if and only if $0\to Q^*\to G^*\to N^*\to 0$ splits. This last exact sequence splits since all terms are free abelian groups. So the sequence of tori is split and hence it has $(*)$ by the Künneth formula.
\end{proof}
\end{lemma}

This finishes step 2 of Strategy \ref{strategy}.

\subsection*{Step 3}

In this step we finish the proof of Theorem \ref{result}. Recall the notion of an almost exact sequence in Definition \ref{almostE}.

\begin{example} Notable examples of almost exact sequences are the following.
\begin{itemize}
\item When $e\to N\to G\to Q\to e$ is an exact sequence of group varieties, the induced sequences $e\to N\lin\to G\lin\to Q\lin\to e$ and $e\to N\ab\to G\ab\to Q\ab\to e$ are almost exact. Checking this is routine and relies on the fact that $N\lin$ has finite index in $N\cap \Glin=\ker(G\lin\to Q\lin)$.
\item When $e\to N\to G\to Q\to e$ is exact and all group varieties are linear, the induced sequences $e\to \Ru(N)\to \Ru(G)\to \Ru(Q)\to e$ and $e\to N\red\to G\red\to Q\red\to e$ are almost exact. Like for the previous sequence, it relies on $\Ru(N)$ having finite index in $N\cap \Ru(G)$.
\item When all group varieties in question are reductive, the sequences\\ $e\to \R(N)\to \R(G)\to \R(Q)\to e$ and $e\to N\semi\to G\semi\to Q\semi\to e$ are almost exact. The same argument as for the other two shows that this is true.
\end{itemize}
\end{example}

We now show that certain almost exact sequences have $(*)$.

\begin{lemma}\label{almost}
For $e\to N\to G\to Q\to e$ an exact sequence of group varieties, the almost exact sequence $e\to N\ab\to G\ab\to Q\ab\to e$ has $(*)$. When all group varieties are linear, the almost exact sequence $\Ru(N)\to \Ru(G)\to \Ru(Q)\to e$ has $(*)$ and when all group varieties are reductive, the almost exact sequences $e\to \R(N)\to \R(G)\to \R(Q)\to e$ and $e\to N\semi\to G\semi\to Q\semi\to e$ have $(*)$.
\begin{proof}
Consider the exact sequence $e\to \frac{N}{N\cap \Glin}\to G\ab \to Q\ab\to e$. It has $(*)$ by Lemma \ref{types}. Note that $N\ab \to \frac{N}{N\cap \Glin}$ is an isogeny and so it induces an isomorphism on cohomology. Hence it follows that $e\to N\ab\to G\ab\to Q\ab\to e$ has $(*)$.\\
In the case of all group varieties being linear, the statement about the cohomology of unipotent groups is trivial since unipotent groups have no higher cohomology.\\
When all group varieties are reductive consider $e\to \R(N)\to \R(G)\to \frac{\R(G)}{\R(N)}\to e$. Since all radicals are tori, the sequence has $(*)$ by Lemma \ref{types}. Since $\frac{R(G)}{R(N)}\to R(Q)$ is an isogeny, whose pullback is an isomorphism, it follows that $e\to \R(N)\to \R(G)\to \R(Q)\to e$ has $(*)$.\\
For the sequence $e\to N\semi\to G\semi\to Q\semi\to e$, note that $e\to \frac{N}{N\cap \R(G)}\to G\semi\to Q\semi\to e$ is exact and that $N\semi\to \frac{N}{N\cap \R(G)}$ is an isogeny, so this follows as in the case of the abelian varieties.
\end{proof}
\end{lemma}

Now we can combine this lemma with several parts of Step 1 to obtain the following.

\begin{lemma}\label{reductive}
An exact sequence $e\to N\to G\to Q\to e$ of reductive group varieties has $(*)$. An induced sequence $e\to N\red\to G\red\to Q\red$ has $(*)$.
\begin{proof}
Let $e\to N\to G\to Q\to e$ be an exact sequence of reductive group varieties. By considering the induced sequences of radicals and semisimple groups and then passing to $l$-adic cohomology, one gets the following commuting diagram:
\begin{equation*}
\begin{tikzcd}
  & 0                                   & 0                                   & 0                                   &             \\
0 & \mathrm{P}\Ho^*(\R(N)) \arrow[l] \arrow[u]   & \mathrm{P}\Ho^*(R(G)) \arrow[l] \arrow[u]   & \mathrm{P}\Ho^*(R(Q)) \arrow[l] \arrow[u]   & 0 \arrow[l] \\
0 & \mathrm{P}\Ho^*(N) \arrow[u] \arrow[l]      & \mathrm{P}\Ho^*(G) \arrow[l] \arrow[u]      & \mathrm{P}\Ho^*(Q) \arrow[l] \arrow[u]      & 0 \arrow[l] \\
0 & \mathrm{P}\Ho^*(N\semi) \arrow[u] \arrow[l] & \mathrm{P}\Ho^*(G\semi) \arrow[u] \arrow[l] & \mathrm{P}\Ho^*(Q\semi) \arrow[u] \arrow[l] & 0 \arrow[l] \\
  & 0 \arrow[u]                         & 0 \arrow[u]                         & 0 \arrow[u]                         &            
\end{tikzcd}
\end{equation*}
We have to show that the middle row is exact. The top and bottom row are exact by Lemma \ref{almost}. All columns are exact by Corollary \ref{reduct}. By applying the Snake lemma twice, we see that $\mathrm{P}\Ho^*(Q)\to \mathrm{P}\Ho^*(G)$ is injective and that $\mathrm{P}\Ho^*(G)\to \mathrm{P}\Ho^*(N)$ is surjective. Since $N\to G\to Q$ is constant, $\mathrm{P}\Ho^*(Q)$ is included in $\ker(\mathrm{P}\Ho^*(G)\to \mathrm{P}\Ho^*(N))$. A simple diagram-chase then yields that this inclusion is in fact an equality.
\end{proof}
\end{lemma}

We now use this lemma to show that Theorem \ref{result} holds in the case that all group varieties are linear.

\begin{lemma}\label{linear}
An exact sequence $e\to N\to G\to Q\to e$ of linear group varieties has $(*)$. Moreover for $e\to N\to G\to Q\to e$ exact, the almost exact sequence $e\to N\lin\to G\lin \to Q\lin\to e$ has $(*)$.
\begin{proof}
The proof follows similarly to the above proof by considering the following commuting diagram, which has exact columns and an exact bottom row by Lemma \ref{reductive}:
\begin{equation*}
\begin{tikzcd}
 & 0  & 0   & 0      \\
0 & \mathrm{P}\Ho^*(N) \arrow[u] \arrow[l]      & \mathrm{P}\Ho^*(G) \arrow[l] \arrow[u]      & \mathrm{P}\Ho^*(Q) \arrow[l] \arrow[u]      & 0 \arrow[l] \\
0 & \mathrm{P}\Ho^*(N\red) \arrow[u] \arrow[l] & \mathrm{P}\Ho^*(G\red) \arrow[u] \arrow[l] & \mathrm{P}\Ho^*(Q\red) \arrow[u] \arrow[l] & 0 \arrow[l] \\
  & 0 \arrow[u]                         & 0 \arrow[u]                         & 0 \arrow[u]                         &            
\end{tikzcd}
\end{equation*}
For the second statement, note that $e\to N\lin \to G\lin\to \frac{G\lin}{N\lin}\to e$ has $(*)$ since all group varieties are linear and that $\frac{G\lin}{N\lin}\to Q\lin$ is an isogeny.
\end{proof}
\end{lemma}

Now we can prove the main result of the paper, which is Theorem \ref{result}.

\begin{theorem*}
An exact sequence of group varieties $e\to N\to G\to Q\to e$ has $(*)$.
\begin{proof}
The proof is almost identical to the proof of Lemma \ref{reductive}. Consider the large commuting diagram:
\begin{equation*}
\begin{tikzcd}
  & 0                                   & 0                                   & 0                                   &             \\
0 & \mathrm{P}\Ho^*(N\lin) \arrow[l] \arrow[u]   & \mathrm{P}\Ho^*(G\lin) \arrow[l] \arrow[u]   & \mathrm{P}\Ho^*(Q\lin) \arrow[l] \arrow[u]   & 0 \arrow[l] \\
0 & \mathrm{P}\Ho^*(N) \arrow[u] \arrow[l]      & \mathrm{P}\Ho^*(G) \arrow[l] \arrow[u]      & \mathrm{P}\Ho^*(Q) \arrow[l] \arrow[u]      & 0 \arrow[l] \\
0 & \mathrm{P}\Ho^*(N\ab) \arrow[u] \arrow[l] & \mathrm{P}\Ho^*(G\ab) \arrow[u] \arrow[l] & \mathrm{P}\Ho^*(Q\ab) \arrow[u] \arrow[l] & 0 \arrow[l] \\
  & 0 \arrow[u]                         & 0 \arrow[u]                         & 0 \arrow[u]                         &            
\end{tikzcd}
\end{equation*}
Again it suffices to show that the middle row is exact. The columns are exact by Theorem \ref{chevalley} of Step 1. The top and bottom rows are exact by Lemma \ref{linear} and Lemma \ref{almost} respectively. Thus we can apply the snake lemma twice and then note that $\mathrm{P}\Ho^*(Q)\subset \ker (\mathrm{P}\Ho^*(G)\to \mathrm{P}\Ho^*(N))$ to obtain that the middle row is exact, which proves the theorem.
\end{proof}
\end{theorem*}

By using the decomposition of group varieties that we have seen throughout this paper, together with the fact that $l$-adic cohomology is isogeny-invariant, we obtian the following result as a corollary.

\begin{corollary}
For $G$ a group variety, to compute $\Ho^*(G)$, it suffices to compute $\Ho^*(H)$, where $H$ is one of: $\G_m$, $\G_a$, a simple abelian variety, a semi-simple simply connected group variety.
\end{corollary}

\begin{remark}
For $H$ an abelian variety, the cohomology is naturally isomorphic to the exterior algebra on $\Ho^1(H)$, which is the tate module of the $l^n$-torsion of the dual $H^\vee$. In particular $\Ho^*(H)$ depends (as a $\Q_l$-algebra) only on $\dim(H)$. \\
For $H$ a semisimple and simply connected group variety, $H$ is either in an isomorphism class $A_{n-1}$, $B_n$, $C_n$ or $D_n$ or $H$ is one of the exceptional groups $G_2,F_4,E_6,E_7$ or $E_8$. For $T$ a maximal torus of $H$, set $S=S=\Sym(X^*(T)\otimes_\Z \Q_l))$ and denote by $W$ the Weyl group of $H$ (see \cite{springer2008linear} Section 7.1 for an introduction on the Weyl group). The cohomology ring $\Ho^*(H)$ is given by $\bigwedge^* J[\times 2-1]$ by (\cite{sga41/2} p.230), where for $S^W$ the Weyl-group invariants of $S$, $J=S^W_+/(S^W_+)^2$, where $S^W_+\subset S^W$ is the ideal of invariants of positive degree and $[\times 2-1]$ indicates that the grading of an element is doubled and then decreased by $1$.\\
Computing these invariants is a nontrivial task in the exceptional cases and it was done for $G_2$ in \cite{G2} and for $F_4,E_6,E_7,E_8$ in \cite{F4E6E7E8}.The algebra structure of $\Ho^*(H)$ is determined by the degrees of these invariants, which are listed in (\cite{freudenthal1969linear}, p.516).

\end{remark}

We remark that the following generalization of Theorem \ref{result} can not be made.\\

\begin{remark}
One might wonder if the condition that $N\subset G$ is normal can be relaxed to $N=H$ being any subgroup variety of $G$. In this case the quotient $G/H$ exists as a variety over $k$ and one can ask whether one has a similar isomorphism $\Ho^*(G)\cong \Ho^*(H)\otimes \Ho^*(G/H)$.\\
This fails already in some very simple cases. For example, we let $n\geq 2$ and take $G=\GL_n$ and the torus $H=\{\diag(t_1,...,t_n)\,|\,t_1,...,t_n\in k^*\}\cong \G_m^n$. We have by (\cite{sga41/2} p.230) an isomorphism of graded $\Q_l$-algebras $\Ho^*(\GL_n)\cong \bigwedge^*\J[\times 2-1]$ , where $\J$ is the $\Q_l$-vector space spanned freely by the primitive Weyl-group invariants of $\GL_n$. In this case the primitive invariants are $e_1,...,e_n$, the elementary symmetric polynomials and $[\times 2-1]$ indicates that the weight of $e_i$ is $2i-1$. Hence $\dim \Ho^1(\GL_n)=1$, while $\dim \Ho^1(\G_m^n)=n$ by the Künneth formula, so indeed the above generalization does not hold.
\end{remark}

We end with the following question for possible further research.

\begin{question*}
The proof of Theorem \ref{result} relies on various structure theorems for algebraic groups and being able to compare certain exact sequences to split exact sequences. One may ask if there is a more intrinsic proof than ours. For example, can one show directly that the $l$-adic Leray spectral sequence of $G\to Q$ has $E_2^{p,q}=\Ho^p(Q,\Q_l)\otimes_{\Q_l}\Ho^q(N,\Q_l)$ and that it degenerates at the $E_2$-page?
\end{question*}

\subsection*{Acknowledgements}
The author wants to thank prof. Gunther Cornelissen for his supervision, enthusiasm and helpful discussions during the research period. The author also wants to thank the anonymous referee for giving helpful comments during the submission process.

\printbibliography
\end{document}